\documentclass[10pt]{article}
\usepackage{amsfonts}
\usepackage{amsthm,amsmath,amssymb,anysize}
\newtheorem{lemma}{Lemma}[section]
\newtheorem{theorem}[lemma]{Theorem}
\newtheorem{remark}[lemma]{Remark}
\newtheorem{proposition}[lemma]{Proposition}

\usepackage[numbers,sort&compress]{natbib}
\marginsize{3.2cm}{3.2cm}{3.6cm}{3cm}
\numberwithin{equation}{section}

\title{\textsf{The minimal dimensions of faithful representations  for Heisenberg Lie superalgebras}}
\author{Wende Liu\footnote{Corresponding author. Email:
\texttt{wendeliu@ustc.edu.cn}}, Meiwei Chen
 \setcounter{footnote}{-1}\footnote{}\\
\small\textit{School of Mathematical Sciences, Harbin Normal University}\\
\small \textit{ Harbin  150025, P.R. China} \\
\small\textit{}\\
}
\date{ }
\begin{document}
\maketitle
\begin{quotation}
\small\noindent \textbf{Abstract}: This paper aims to determine the minimal dimensions and super-dimensions of faithful representations  for   Heisenberg Lie superalgebras over an algebraically closed field of characteristic zero.

 \vspace{0.2cm} \noindent{\textbf{Keywords}}: Heisenberg Lie superalgebra; faithful representation; minimal dimension

\end{quotation}
\setcounter{section}{0}
\section{Introduction}

Throughout $\mathbb{F}$ is an algebraically closed field of characteristic zero and all vector spaces and algebras are over $\mathbb{F}$ and of finite
dimensions.

Ado's theorem says that every finite dimensional Lie (super)algebra has a finite-dimensional faithful representation \cite{VK}. Let $\frak{g}$ be a Lie (super)algebra and write
$$
\mu(\mathfrak{g})=\min \{\dim V\mid   V \mbox{ is a faithful $\mathfrak{g}$-module}\}.
$$
It is in general  difficult to determine $\mu(\mathfrak{g})$. The earliest result is that $\mu(\mathfrak{g})=\lceil 2\sqrt{\dim \mathfrak{g}-1} \rceil$ for an abelian Lie algebra $\mathfrak{g}$, which is due to Schur for $\mathbb{F}=\mathbb{C}$ and to Jacobson  for arbitrary $\mathbb{F}$ (see also \cite{Mirzakhani}, for a simple proof due to  Mirzakhani).
In 1998 Burde concluded that $\mu(\mathfrak{h}_{m})=m+2$  for Heisenberg Lie algebra  $\mathfrak{h}_{m}$  of dimension $2m+1$   \cite{DB}.
In 2008 Burde and Moens established an explicit formula of $\mu(\mathfrak{g})$ for  semi-simple and reductive Lie algebras \cite{DW}.
 In 2009 Cagliero and Rojas obtained a formula $\mu(\mathfrak{h}_{m, p})$ for the current Heisenberg Lie algebra $\mathfrak{h}_{m, p}$ \cite{LN}. One can also find the formula  $\mu(\mathfrak{J})$   for a Jordan algebra $\mathfrak{J}$ with the trivial multiplication \cite{WS}.

However, very little is known about the function $\mu$ for  Lie superalgebras. In 2012   Liu and Wang determined
$\mu(\mathfrak{g})=\lceil 2\sqrt{\dim \mathfrak{g}} \rceil$ for  any purely odd Lie superalgebra $\mathfrak{g}$ \cite{WS} and it remains open   to determine $\mu(\mathfrak{g})$ for an abelian Lie superalgebra $\frak{g}$ with nontrivial even part.
In this paper, we shall determine the minimal (super-)dimensions of the faithful representations for Heisenberg Lie superalgebras.

A two-step nilpotent Lie superalgebra with $1$-dimensional center is called a Heisenberg Lie superalgebra. Then
Heisenberg Lie superalgebras split into the following two types according to the parities of their centers   \cite{MGO}.
Write $\mathfrak{h}_{m, n}$ for the Heisenberg Lie superalgebra with 1-dimensional even center $\mathbb{F}z$, which has a $\mathbb{Z}_{2}$-homogeneous basis
 \begin{eqnarray*}
 (u_{1},\ldots,u_{m},v_{1},\ldots,v_{m};z\mid w_{1},\ldots,w_{n})
 \end{eqnarray*}
with multiplication given by
$$[u_{i}, v_{i}]=-[v_{i}, u_{i}]=z=[w_{j}, w_{j}], \quad i=1,\ldots,m,  j=1,\ldots,n,$$
the remaining brackets being zero. Hereafter $\mathbb{Z}_{2}=\{\bar{0}, \bar{1}\}$ is the group of order $2$.

Write $\mathfrak{h}_{n}$ for the Heisenberg Lie superalgebra with 1-dimensional odd center $\mathbb{F}z$, which has a $\mathbb{Z}_{2}$-homogeneous basis
 \begin{eqnarray*}
 (v_{1},\ldots,v_{n}\mid z; w_{1},\ldots,w_{n})
 \end{eqnarray*}
with multiplication given by
$$[v_{i}, w_{i}]=z=-[w_{i}, v_{i}], i=1,\ldots,n,$$
the remaining brackets being zero.

Both $\mathfrak{h}_{m, n}$ and $\mathfrak{h}_{n}$ are nilpotent.
 Note that $\mathfrak{h}_{m, 0}$ is a Heisenberg Lie algebra and $\mathfrak{h}_{0, n}$ is isomorphic to the Heisenberg Lie superalgebra considered in \cite[p.18]{VK}, whose even part coincides with 1-dimension center.
However, the Heisenberg Lie superalgebras with odd centers, $\mathfrak{h}_{n}$,  have no  analogs  in Lie algebras. We should also mention that  Hegazi studied  representations of the Heisenberg Lie superalgebras of even center, $\mathfrak{h}_{m, n}$,  and tried to find a finite-dimensional faithful representation of $\mathfrak{h}_{m, n}$ \cite[\S 3]{Hegazi}.

Throughout this paper,  subalgebras and  (sub)modules of Lie superalgebras are assumed to be $\mathbb{Z}_2$-graded. Hereafter we write $\frak{g}$ for $\mathfrak{h}_{m,n}$ or $\mathfrak{h}_{n}$. A main result of this paper is that
\begin{eqnarray*}
\mu(\frak{g})=\left\{\begin{array}{ll}
  m+\lceil n/2\rceil+2 & \frak{g}=\frak{h}_{m,n}\\
n+2 &\frak{g}=\frak{h}_{n}.
\end{array}\right.
\end{eqnarray*}
To formulate the super-dimensions of the faithful representations, write for  $i\in \{0,1\}$,
$$
\mu_{i}(\mathfrak{g})=\min \{\dim V_{\bar{i}}\mid  \mbox{$V$ is a faithful $\mathfrak{g}$-module}\};
$$
$$
\mu_{i}^{*}(\mathfrak{g})=\min \{\dim V\mid \mbox{$V$ is a faithful $\mathfrak{g}$-module with} \dim V_{\bar{i}}=\mu_{i}(\mathfrak{g})  \}.
$$
In this paper we also determine the values  $\mu_{i}(\mathfrak{g})$ and $\mu_{i}^{*}(\mathfrak{g})$.

\section{Minimal dimensions}

Since Engel's theorem holds for Lie superalgebras, as in Lie algebra case \cite[Lemma 1]{DB}, we have
\begin{lemma}\label{lem1}
Let $L$ be a nilpotent Lie superalgebra with a $1$-dimensional  center $\mathbb{F}z $.
Then a representation $\lambda: L\rightarrow \mathfrak{gl}(V)$ is faithful if and only if $z$ acts nontrivially.
\end{lemma}

\begin{proof} The ``only if" part is obvious. Suppose $z$ acts nontrivially.
If $\ker(\lambda)\neq 0$, then Engel's theorem ensures that $\ker(\lambda)$ contains a nonzero element killed by $L$ and hence $\ker(\lambda)$ contains the center $\mathbb{F}z$, showing that $\rho(z)=0,$
  a contradiction.
\end{proof}

Let
$$
\zeta(\frak{g})=\max\{\dim \frak{a}\mid \mbox{$\frak{a}$ is an abelian subalgebra of $\frak{g}$ not containing the center of $\frak{g}$}\}.
$$

 Let $\sqrt{-1}$ denote a fixed root of the equation $x^{2}=-1$ in $\mathbb{F}$. We have
\begin{lemma}\label{lem2}
Let $\mathfrak{a}$ be an abelian subalgebra not containing $z$ of $\mathfrak{g}$ and having dimension $\zeta(\frak{g})$.
Then
 \begin{itemize}
 \item for $\mathfrak{g}=\mathfrak{h}_{m, n}$, the super-dimension $(\dim \mathfrak{a}_{\bar{0}},\dim \mathfrak{a}_{\bar{1}})$ must be $(m,\lfloor n/2 \rfloor);$
 \item for $\mathfrak{g}=\mathfrak{h}_{n}$, the super-dimension $(\dim \mathfrak{a}_{\bar{0}},\dim \mathfrak{a}_{\bar{1}})$ has $n+1$ possibilities:
 $$ (i,n-i),\quad i=0,  \ldots, n.$$
 \end{itemize}
 In particular,
\begin{eqnarray*}
\zeta(\frak{g})=\left\{\begin{array}{ll}
m+\lfloor  n/2 \rfloor   & \frak{g}=\frak{h}_{m,n}\\
n &\frak{g}=\frak{h}_{n}.
\end{array}\right.
\end{eqnarray*}
\end{lemma}
\begin{proof}
Since  $\mathfrak{a}$  does not contain the center $\mathbb{F}z$, there is a $\mathbb{Z}_{2}$-graded subspace $\frak{k}$ containing $\frak{a}$ such that
  $\mathfrak{g}=\mathfrak{k}\oplus \mathbb{F}z$.
  Let $B:\mathfrak{k}\times \mathfrak{k}\rightarrow \mathbb{F}$ be the  form determined  by $[x, y]=B(x, y)z$ for all $x, y\in \mathfrak{k}$.
It is clear that $B$ is bilinear and  non-degenerate.
Since $\mathfrak{a}$ is abelian,
  $B(x, y)=0$  for all $x, y\in \mathfrak{a}$.
Therefore, $\mathfrak{a}$ is a $B$-isotropic subspace of $\mathfrak{k}$. It follows that
$\dim \mathfrak{a}\leq \frac{\dim \mathfrak{k}}{2}=\frac{\dim \mathfrak{g}-1}{2}$.

Suppose $\mathfrak{g}=\frak{h}_{m,n}.$ Then $\dim \mathfrak{a}\leq m+\lfloor n/2 \rfloor$. Let $\mathfrak{b}$ be the subspace spanned by
$$ u_{1},u_{2},\ldots,u_{m},w_{1}+\sqrt{-1}w_{2},w_{3}+\sqrt{-1}w_{4},\ldots,w_{n-1}+\sqrt{-1}w_{n}$$ if $n$ is even
and by
$$u_{1},u_{2},\ldots,u_{m},w_{1}+\sqrt{-1}w_{2},w_{3}+\sqrt{-1}w_{4},\ldots,w_{n-2}+\sqrt{-1}w_{n-1}$$ if $n$ is odd.
One can check that $\mathfrak{b}$ is an abelian subalgebra of dimension $m+\lfloor n/2 \rfloor$ and $\mathfrak{b}$ does not contain $z$.
Hence, $\zeta(\frak{g})=\dim \mathfrak{a}=m+\lfloor n/2 \rfloor$.

Clearly,
$\mathfrak{a}_{\bar{0}}$ is a $B$-isotropic subspace of $\mathfrak{k}_{\bar{0}}$ and
$\mathfrak{a}_{\bar{1}}$ is a $B$-isotropic subspace of $\mathfrak{k}_{\bar{1}}$.
Since $B|_{\mathfrak{k}_{\bar{0}}\times \mathfrak{k}_{\bar{0}}}$ and $B|_{\mathfrak{k}_{\bar{1}}\times \mathfrak{k}_{\bar{1}}}$ are non-degenerate, we have $\dim \mathfrak{a}_{\bar{0}}\leq m$,
$\dim \mathfrak{a}_{\bar{1}}\leq \lfloor n/2 \rfloor$.
Note that $\dim \mathfrak{a}=m+\lfloor n/2 \rfloor$. It follows that $\dim \mathfrak{a}_{\bar{0}}=m$, $\dim \mathfrak{a}_{\bar{1}}=\lfloor n/2 \rfloor$.

Suppose $\mathfrak{g}=\frak{h}_{n}.$ Then $\dim \mathfrak{a}\leq n$.
Let $\frak{b}'$ be the subspace spanned by
$v_{1},v_{2},\ldots,v_{n}.$
Clearly, $\frak{b}'$ is an abelian subalgebra of dimension $n$ of $\mathfrak{h}_{n}$ and $\frak{b}'$ does not contain $z$.
Hence, $\zeta(\frak{g})=\dim \mathfrak{a}=n.$
From the definition of  $\mathfrak{h}_{n}$, one may easily find abelian subalgebras not containing $z$ and having the indicated super-dimension $(i,n-i)$  with
$ i=0,  \ldots, n.$
\end{proof}

\begin{lemma}\label{lem2047}
Let $V$ be a faithful $\mathfrak{g}$-module. Then there exists a nonzero homogeneous element $v_{0}$ in $V$ such that $zv_{0}\neq0$. Moreover, let $\rho_{v_{0}}$ be the linear mapping defined by
\begin{eqnarray*}
\rho_{v_{0}}: \mathfrak{g} \longrightarrow V,\quad  x\longmapsto xv_{0}
\end{eqnarray*}
and let $\mathfrak{a}=\ker(\rho_{v_{0}})$ and $V_0=\mathrm{im}(\rho_{v_{0}})$.
Then $\mathfrak{a}$ is an abelian subalgebra not containing $z$ and
if $\dim \mathfrak{a}=\zeta(\frak{g})$, then $v_{0}\notin V_0$.
\end{lemma}

\begin{proof}  Lemma \ref{lem1} ensures  that there exists a nonzero homogeneous element $v_{0}$ in $V$ such that $zv_{0}\neq0$. It follows that $\mathfrak{a}$ does not contain $z$.
Since $\rho_{v_{0}}$ is homogenous, $\mathfrak{a}$ is a $\mathbb{Z}_{2}$-graded subspace of $\mathfrak{g}$.
 For  $x, y\in \mathfrak{a}$, it is obvious that
$[x, y]\in \mathfrak{a}\cap  \mathbb{F}z=0$  and it follows   that $\mathfrak{a}$ is an abelian subalgebra.

 Suppose $\dim \mathfrak{a}=\zeta(\frak{g})$. Assume in contrary that $v_{0}\in V_0$. Then there exists an $x\in \mathfrak{g}_{\bar{0}}$ such that $xv_{0}=v_{0}$,
since $v_{0}$ is a nonzero homogeneous element of $V$.
Clearly, $(\mathfrak{h}_{m, n})_{\bar{0}}$ is a solvable Lie algebra.
Since $[u_{i}, v_{i}]=z$,
by Lie's theorem, $z$ acts nilpotently on $V$. For $\mathfrak{h}_{n}$, $z$ is odd.
Therefore, $x\notin \mathbb{F}z$. Moreover, it is clear that $x\notin \frak{a}$.
Then by the maximality of $\mathfrak{a}$, we have $[x, \mathfrak{a}]\not=0$. There must be some $y\in \mathfrak{a}$ such that $[x, y]=z$.
Since  $x\in \mathfrak{g}_{\bar{0}}$, we have
\begin{equation*}\label{eq21}
zv_{0}=[x, y]v_{0}
=x(yv_{0})-y(xv_{0})=0,
\end{equation*}
using that $yv_{0}=0$ and $xv_{0}=v_{0}$.
This is a contradiction. Hence $v_{0}\notin V_0$.
\end{proof}

\begin{proposition}\label{proposition2}
Let $\frak{g}=\mathfrak{h}_{m, n}$ or $\mathfrak{h}_{n}$. Then
\begin{eqnarray*}
\mu(\frak{g})\geq \dim \mathfrak{g}-\zeta(\frak{g})+1.
\end{eqnarray*}
That is,
\begin{itemize}
\item
 $\mu(\mathfrak{h}_{m, n})\geq m+\lceil n/2 \rceil+2;$
 \item $\mu(\mathfrak{h}_{n})\geq n+2.$
 \end{itemize}
\end{proposition}
\begin{proof}
Assume that $\lambda: \mathfrak{g}\rightarrow \mathfrak{gl}(V)$ is a faithful representation. Let $v_{0}, \mathfrak{a}, V_0$ be as in Lemma \ref{lem2047}.
By Lemmas \ref{lem2} and \ref{lem2047}, we have
$$
\dim V\geq \dim V_0=\dim \mathfrak{g}-\dim \mathfrak{a}\geq \dim \mathfrak{g}-\zeta(\frak{g}).
$$
 If $\dim V_0\geq \dim \mathfrak{g}-\zeta(\frak{g})+1$, we are done.
Suppose $\dim V_0= \dim \mathfrak{g}-\zeta(\frak{g}).$ Then $\dim \mathfrak{a}=\zeta(\frak{g})$.
By Lemma \ref{lem2047}, we have $v_{0}\notin V_0$.
Therefore,
$$\dim V\geq \dim V_0+1= \dim \mathfrak{g}-\zeta(\frak{g})+1.$$
That is,
 $\mu(\mathfrak{h}_{m, n})\geq m+\lceil n/2 \rceil+2;$
 $\mu(\mathfrak{h}_{n})\geq n+2.$
\end{proof}

\begin{theorem}\label{theorem1} We have
\begin{eqnarray*}
\mu(\frak{g})=\left\{\begin{array}{ll}
  m+\lceil n/2\rceil+2 & \frak{g}=\frak{h}_{m,n}\\
n+2 &\frak{g}=\frak{h}_{n}.
\end{array}\right.
\end{eqnarray*}
\end{theorem}
\begin{proof}
By Proposition \ref{proposition2}, it is enough to establish a faithful representation of the desired dimension for $\frak{g}$.
Consider the even linear mapping
$$\pi: \mathfrak{h}_{m, n}\longrightarrow \mathfrak{gl}(m+2\mid\lceil n/2 \rceil)$$
given by
\begin{eqnarray*}
&&\pi(u_{i})=e_{1,i+1},\quad
 \pi(v_{i})=e_{i+1,m+2},\quad \pi(z)=e_{1,m+2},\\
&&\pi(w_{2k-1})=\frac{1}{2}e_{m+2+k,m+2}+ e_{1,m+2+k},\\
&&\pi(w_{2k})=\frac{\sqrt{-1}}{2}e_{m+2+k,m+2}- \sqrt{-1}e_{1,m+2+k},
\end{eqnarray*}
where $1\leq i\leq m, 1\leq 2k, 2k-1\leq n.$
Under $\pi$,
 an element of $\frak{h}_{m,n},$
\begin{eqnarray}\label{eq1013w}
\sum^{m}_{i=1}a_{i}u_{i}+\sum^{m}_{i=1}b_{i}v_{i}+cz+\sum^{n}_{j=1}d_{j}w_{j}\quad (\mbox{$a_{i}, b_{i}, c, d_{j}\in \mathbb{F}$})
\end{eqnarray}
 is  presented as
\begin{equation}\label{eq12}
\left(\begin{array}{cccccc|cccc}
0&a_{1}&a_{2}&\cdots&a_{m}&c                   &d_{1,2}&d_{3,4}&\cdots&d_{n-1,n}\\
 &     &     &      &     &b_{1}               &                                           \\
 &     &     &      &     &b_{2}               &                                            \\
 &     &     &      &     &\vdots              &      \\
 &     &     &      &     &b_{m}               &        \\
 &     &     &      &     &0                   &              \\
 \hline
 &     &     &      &     &\widetilde{d}_{1,2}  &        \\
 &     &     &      &     &\widetilde{d}_{3,4}  &              \\
 &     &     &      &     &\vdots              &      \\
 &     &     &      &     &\widetilde{d}_{n-1,n}&              \\
 \end{array}\right)\quad (\mbox{$n$ even})
 \end{equation}
 or
\begin{equation}\label{eq13}
\left(\begin{array}{cccccc|ccccc}
0&a_{1}&a_{2}&\cdots&a_{m}&c&d_{1,2}&d_{3,4}&\cdots&d_{n-2,n-1}&d_{n}\\
 &     &     &      &     &b_{1}&                                           \\
 &     &     &      &     &b_{2}&                                            \\
 &     &     &      &     &\vdots&      \\
 &     &     &      &     &b_{m}&        \\
 &     &     &      &     &0&              \\\hline
 &     &     &      &   &\widetilde{d}_{1,2}&        \\
 &     &     &      &   &\widetilde{d}_{3,4}&              \\
 &     &     &      &     &\vdots&      \\
 &     &     &      &   &\widetilde{d}_{n-2,n-1}&              \\
 &     &     &      &   &\frac{1}{2}d_{n}&              \\
                           \end{array}\right)\quad (\mbox{$n$  odd}),
 \end{equation}
where $d_{i,i+1}=d_{i}-\sqrt{-1}d_{i+1}$, $\widetilde{d}_{i,i+1}=\frac{1}{2}(d_{i}+\sqrt{-1}d_{i+1}).$
It is routine to verify that $\pi$ is a faithful representation of dimension $ m+\lceil n/2\rceil+2$.

Let us consider the even linear mapping
$$\pi':\mathfrak{h}_{n}\longrightarrow \mathfrak{gl}(n+1\mid 1)$$
given by
\begin{eqnarray*}
&&\pi'(v_{i})=e_{1,i+1},\\
&&\pi'(z)=e_{1,n+2},\quad
 \pi'(w_{i})=e_{i+1,n+2},
\end{eqnarray*}
where $1\leq i\leq n.$
Under $\pi'$,
an element of $\frak{h}_{n}$,
\begin{eqnarray}\label{eq1057w}
\sum^{n}_{i=1}a_{i}v_{i}+cz+\sum^{n}_{i=1}b_{i}w_{i}\quad (\mbox{$a_{i}, c, b_{i}\in \mathbb{F}$})
\end{eqnarray}
  is presented as
\begin{equation*}\label{eq14}
\left(\begin{array}{ccccc|c}
0&a_{1}&a_{2}&\cdots&a_{n}&c\\
 &     &     &      &     &b_{1}\\
 &     &     &      &     &b_{2}\\
 &     &     &      &     &\vdots       \\
 &     &     &      &     &b_{n}        \\ \hline
 &     &     &      &     &0            \\
                           \end{array}\right).
                           \end{equation*}
It is routine to verify that $\pi'$ is a faithful representation of dimension $n+2$.
\end{proof}

\section{Super-dimensions}

In this section  we discuss the super-dimensions of the faithful representations for Heisenberg Lie superalgebras. We first establish a technical lemma, for which we shall use a result due to Burde  \cite{DB}: the formula $\mu(L)$ for Heisenberg Lie algebras.
\begin{lemma}\label{proposition5}
Let $V$ be a faithful module of  $\mathfrak{h}_{m,n}$. Let $v_0$ be as in Lemma \ref{lem2047}.
If $v_{0}$ is even, then $\dim V_{\bar{0}}\geq m+2$; if $v_{0}$ is odd, then $\dim V_{\bar{1}}\geq m+2$.
\end{lemma}
\begin{proof}Note that $(\mathfrak{h}_{m, n})_{\bar{0}}$ is a Heisenberg Lie algebra.
Obviously, $V_{\bar{0}}$ is a module of the Lie algebra $(\mathfrak{h}_{m, n})_{\bar{0}}$.
If $v_{0}$ is even, then $v_{0}\in V_{\bar{0}}$. Since $zv_{0}\neq 0$,
$V_{\bar{0}}$ is a faithful module of $(\mathfrak{h}_{m, n})_{\bar{0}}$ by Lemma \ref{lem1}. According to the minimal dimensions of  faithful representations for Heisenberg Lie algebras \cite{DB}, we have
$\dim V_{\bar{0}}\geq \mu((\mathfrak{h}_{m, n})_{\bar{0}})=m+2$.
Similarly, if $v_{0}$ is odd, then $V_{\bar{1}}$ is a faithful module of $(\mathfrak{h}_{m, n})_{\bar{0}}$ and hence $\dim V_{\bar{1}}\geq m+2$.
\end{proof}

\begin{theorem}\label{corollary4}
Suppose $V$ is a faithful $\mathfrak{g}$-module of the minimal dimension $\mu(\mathfrak{g})$. Then
 \begin{itemize}
 \item For $\mathfrak{h}_{m, n}$, the super-dimension $(\dim V_{\bar{0}},\dim V_{\bar{1}})$ has $2$ possibilities:
 $$\mbox{$(m+2,\lceil n/2 \rceil)$, $(\lceil n/2 \rceil,m+2)$;}$$
 \item For $\mathfrak{h}_{n}$, the super-dimension $(\dim V_{\bar{0}},\dim V_{\bar{1}})$ has $n+1$ possibilities:
 $$ (i+1,n-i+1),\quad i=0,  \ldots, n.$$
 \end{itemize}
\end{theorem}

\begin{proof}
Let $v_{0}, \mathfrak{a}, V_0$ be as in Lemma \ref{lem2047}. Since $\mathfrak{a}$ does not contain the center $\mathbb{F}z$, there exists a subalgebra $\mathfrak{a}'$ containing $z$ such that $\mathfrak{g}=\mathfrak{a}\oplus \mathfrak{a}'.$
Since $\dim V=\dim \mathfrak{g}-\zeta(\frak{g})+1$ and $\dim \mathfrak{a}\leq \zeta(\frak{g})$,
we have $\dim \mathfrak{g}-\zeta(\frak{g}) \leq \dim V_0 \leq \dim \mathfrak{g}-\zeta(\frak{g})+1$. It is enough to consider the following two cases.
\\

\noindent\textit{Case 1}: $\dim V_0= \dim \mathfrak{g}-\zeta(\frak{g})$. Then $\dim \mathfrak{a}= \zeta(\frak{g})$
and Lemma \ref{lem2047}  yields $v_{0}\notin V_0$. Then we have $\dim \mathfrak{a}'= \dim \mathfrak{g}-\zeta(\frak{g})$.
Since $\dim V= \dim \mathfrak{g}-\zeta(\frak{g})+1$,  it easy to see that  $V$ has an $\mathbb{F}$-basis
\begin{eqnarray}\label{eqlc1533}
\{v_{0}, xv_{0}\mid  \mbox{$x$ runs over a homogeneous basis of $\mathfrak{a'}$}\}.
\end{eqnarray}

For $\frak{g}=\frak{h}_{m,n}$, by  Lemma \ref{lem2} we have $\dim \mathfrak{a}_{\bar{0}}=m$ and $\dim \mathfrak{a}_{\bar{1}}=\lfloor n/2 \rfloor$.
Hence, $\dim \mathfrak{a}'_{\bar{0}}=m+1$,
 $\dim \mathfrak{a}'_{\bar{1}}=\lceil n/2 \rceil$.
By (\ref{eqlc1533}), if $v_{0}\in V_{\bar{0}}$  then $\dim V_{\bar{0}}= m+2$ and $\dim V_{\bar{1}}= \lceil n/2 \rceil$;
if $v_{0}\in V_{\bar{1}}$, then $\dim V_{\bar{0}}= \lceil n/2 \rceil$ and  $\dim V_{\bar{1}}= m+2 $.

For $\frak{g}=\frak{h}_{n}$,  by Lemma \ref{lem2}, $\dim \mathfrak{a}_{\bar{0}}=i$ and $\dim \mathfrak{a}_{\bar{1}}=n-i,$ $ i=0, \ldots, n.$
Hence, $\dim \mathfrak{a}'_{\bar{0}}=i$ and $\dim \mathfrak{a}'_{\bar{1}}=n+1-i,$ $ i=0, \ldots, n.$
Therefore we have $\dim V_{\bar{0}}= i+1 $ and $\dim V_{\bar{1}} =n+1-i$, where $i=0,\ldots, n$.
\\

\noindent\textit{Case 2}: $\dim V_0=\dim \mathfrak{g}-\zeta(\frak{g})+1$. Then $\dim \mathfrak{a}=  \zeta(\frak{g})-1$  and
$\dim \mathfrak{a}'= \dim \mathfrak{g}-\zeta(\frak{g})+1$.
Since $\dim V= \dim \mathfrak{g}-\zeta(\frak{g})+1$, one sees that $V$ has an $\mathbb{F}$-basis
\begin{eqnarray}\label{mwjxe1}
\{xv_{0}\mid  \mbox{$x$ runs over a homogeneous basis of $\mathfrak{a}'$}\}.
\end{eqnarray}

For $\frak{g}=\frak{h}_{m,n}$,
clearly, $\dim \mathfrak{a}'_{\bar{0}}=m+i$ and
 $\dim \mathfrak{a}'_{\bar{1}}=\lceil n/2 \rceil+2-i$ for some $i\in \{1, 2\}$.
By (\ref{mwjxe1}), if $v_{0}\in V_{\bar{0}}$, then $\dim V_{\bar{0}}=m+i$ and
 $\dim V_{\bar{1}}=\lceil n/2 \rceil+2-i$; if $v_{0}\in V_{\bar{1}}$,
 then $\dim V_{\bar{0}}=\lceil n/2 \rceil+2-i$ and
 $\dim V_{\bar{1}}=m+i$ for some $i\in \{1, 2\}$.
By Lemma \ref{proposition5}, it must be $i=2$.

For $\frak{g}=\frak{h}_{n}$, then $\dim \mathfrak{a}=n-1$. Clearly, $\dim \mathfrak{a'}=n+2$, $\dim \mathfrak{a}'_{\bar{0}}=i+1$ and $\dim \mathfrak{a}'_{\bar{1}}=n+1-i,$ $ i=0, \ldots, n-1$. Therefore, we have either $\dim V_{\bar{0}}= i+1 $ and $\dim V_{\bar{1}} =n+1-i$, or
$\dim V_{\bar{0}}=n+1-i $ and $\dim V_{\bar{1}} = i+1$, for some $i\in \{0,\ldots, n-1\}$.
\\

Up to now, we have shown that:
 \begin{itemize}
 \item For $\mathfrak{h}_{m, n}$, the super-dimension $(\dim V_{\bar{0}},\dim V_{\bar{1}})$ has at most $2$ possibilities:
 $$\mbox{$(m+2,\lceil n/2 \rceil)$, $(\lceil n/2 \rceil,m+2)$;}$$
 \item For $\mathfrak{h}_{n}$, the super-dimension $(\dim V_{\bar{0}},\dim V_{\bar{1}})$ has at most $n+1$ possibilities:
 $$ (i+1,n-i+1),\quad i=0,  \ldots, n.$$
 \end{itemize}

Next let us realize the faithful representations of the super-dimensions indicated above.
For $\mathfrak{h}_{m, n}$, (\ref{eq12}) and (\ref{eq13}) give a minimal faithful representation of $\mathfrak{h}_{m, n}$ with super-dimension $(m+2, \lceil n/2 \rceil)$.
Consider the even linear mapping
$$\pi: \mathfrak{h}_{m, n}\longrightarrow \mathfrak{gl}(\lceil n/2 \rceil\mid m+2)$$
given by
\begin{eqnarray*}
&&\pi(u_{i})=e_{\lceil n/2 \rceil+1,\lceil n/2 \rceil+i+1},\quad
 \pi(v_{i})=e_{\lceil n/2 \rceil+i+1,\lceil n/2 \rceil+m+2},\quad \pi(z)=e_{\lceil n/2 \rceil+1,\lceil n/2 \rceil+m+2},\\
&&\pi(w_{2k-1})=\frac{1}{2}e_{k,\lceil n/2 \rceil+m+2}+ e_{\lceil n/2 \rceil+1,k}, \quad\pi(w_{2k})=\frac{\sqrt{-1}}{2}e_{k,\lceil n/2 \rceil+m+2}- \sqrt{-1}e_{\lceil n/2 \rceil+1,k},
\end{eqnarray*}
where $1\leq i\leq m, 1\leq 2k, 2k-1\leq n.$
Under $\pi$, an element of form (\ref{eq1013w})
 is presented as
\begin{equation}\label{eq61}
\left(\begin{array}{cccc|cccccc}
        &     &     &      &     &     & & & &  \widetilde{d}_{1,2}       \\
        &     &     &      &     &     & & & &  \widetilde{d}_{3,4}       \\
        &     &     &      &     &     & & & &  \vdots                    \\
        &     &     &      &     &     & & & &  \widetilde{d}_{n-1,n}     \\ \hline
d_{1,2} &d_{3,4}&\cdots &d_{n-1,n} &   0    &a_{1}&a_{2}&\cdots&a_{m}&c       \\
        &       &     &    &             &  &     &     &           &b_{1}             \\
        &       &     &    &                         &  &     &     &      &     \vdots     \\
        &       &     &    &            &    &     &      &     &b_{m}             \\
        & &     &     &      &     &   & & & 0\\
        & & & & & & & & &
 \end{array}\right)\quad (\mbox{$n$ even})
 \end{equation}
or
\begin{equation}\label{eq62}
\left(\begin{array}{ccccc|ccccc}
        &     &     &      &     &     & &  & & \widetilde{d}_{1,2}       \\
        &     &     &      &     &     & &  &  &\widetilde{d}_{3,4}       \\
        &     &     &      &     &     & &  &  &\vdots                    \\
        &     &     &      &     &     & &  & & \widetilde{d}_{n-1,n}     \\
        &&&&&&&&&                                \frac{1}{2}d_{n}  \\                           \hline
d_{1,2} &d_{3,4}&\cdots &d_{n-1,n}&d_{n} &   0    &a_{1}&\cdots&a_{m}&c       \\
        &       &     &    &          &   &  &     &              &b_{1}             \\
        &&&&&&&&& b_{2}\\
        &       &     &    &          &               &       &     &      &     \vdots     \\
        &       &     &    &         &   &    &           &     &b_{m}             \\
        & &     &     &      &   &  &   & &  0
 \end{array}\right)\quad (\mbox{$n$  odd}),
 \end{equation}
where $d_{i,j}=d_{i}-\sqrt{-1}d_{j}$, $\widetilde{d}_{i,j}=\frac{1}{2}(d_{i}+\sqrt{-1}d_{j}).$
It is routine to verify that $\pi$ is a faithful representation with  super-dimension $(\lceil n/2\rceil, m+2)$.

For $0\leq r\leq n$, let us consider the even linear mapping
$$\pi':\mathfrak{h}_{n}\longrightarrow \mathfrak{gl}(r+1\mid n-r+1)$$
given by
\begin{eqnarray*}
&&\pi'(v_{i})=e_{1,i+1}, \quad
 \pi'(v_{j})=-e_{j+1,n+2}, \\
&&\pi'(z)=e_{1,n+2},\quad
 \pi'(w_{k})=e_{k+1,n+2}, \quad
\pi'(w_{l})=e_{1,l+1},
\end{eqnarray*}
where $1\leq i, k\leq r$ and $r+1\leq j,l\leq n.$
Under $\pi'$, an element (\ref{eq1057w}) of $\mathfrak{h}_{n}$ is presented as
\begin{equation}\label{eq71}
\left(\begin{array}{cccc|cccc}
0&a_{1}&\cdots&a_{r}&b_{r+1}&\cdots&b_{n}&c\\
 &     &      &     &       &      &     &b_{1}\\
 &     &      &     &       &      &     &\vdots        \\
 &     &      &     &       &      &     &b_{r}        \\ \hline
 &     &      &     &       &      &     &-a_{r+1}\\
 &     &      &     &       &      &     &\vdots       \\
 &     &      &     &       &      &     &-a_{n}        \\
 &     &      &     &       &      &     &0
\end{array}\right).
\end{equation}
It is routine to verify that $\pi'$ is a faithful representation  with  super-dimension $(r+1,  n-r+1)$ for all $r= 0,\ldots, n$.
\end{proof}

Recall that for  $i\in \{0,1\}$,
$$
\mu_{i}(\mathfrak{g})=\min \{\dim V_{\bar{i}}\mid  \mbox{$V$ is a faithful $\mathfrak{g}$-module}\},
$$
$$
\mu_{i}^{*}(\mathfrak{g})=\min \{\dim V\mid   \mbox{$V$ is a faithful $\mathfrak{g}$-module with} \dim V_{\bar{i}}=\mu_{i}(\mathfrak{g})\}.
$$

\begin{theorem}\label{corollary5}   We have
\begin{eqnarray*}
\mu_{0}(\frak{g})=\mu_{1}(\frak{g})=\left\{\begin{array}{ll}
  \min\{m+2, \lceil n/2 \rceil\} & \frak{g}=\frak{h}_{m,n}\\
1 &\frak{g}=\frak{h}_{n}
\end{array}\right.
\end{eqnarray*}
and
\begin{eqnarray*}
\mu_{0}^{*}(\mathfrak{g})=\mu_{1}^{*}(\mathfrak{g})=\left\{\begin{array}{ll}
  m+\lceil n/2 \rceil+2 & \frak{g}=\frak{h}_{m,n}\\
n+2 &\frak{g}=\frak{h}_{n}.
\end{array}\right.
\end{eqnarray*}
\end{theorem}

\begin{proof}
Let $(\lambda,V)$ be a faithful representation of $\mathfrak{g}$.
Evidently,
\begin{equation}\label{10}
\mu_{0}^{*}(\mathfrak{g})\geq \mu(\mathfrak{g});  \mu_{1}^{*}(\mathfrak{g})\geq \mu(\mathfrak{g}).
\end{equation}
Keep the notations in Lemma \ref{lem2047}. As in the proof of Theorem \ref{corollary4}, there exists a subalgebra $\mathfrak{a}'$ containing $z$ such that $\mathfrak{g}=\mathfrak{a}\oplus \mathfrak{a}'.$
By Lemma \ref{lem2},  $\dim \mathfrak{a}'\geq \dim \mathfrak{g}-\zeta(\mathfrak{g})$.
Hence, by Lemma \ref{lem2047}(4), if $v_{0}$ is even, then $\dim V_{\bar{1}}\geq \dim \mathfrak{a}'_{\bar{1}}$;
if $v_{0}$ is odd, then $\dim V_{\bar{0}}\geq \dim \mathfrak{a}'_{\bar{1}}$.

Let $\mathfrak{g}=\mathfrak{h}_{m, n}$. By Lemma \ref{lem2}, we have $\dim \mathfrak{a}'_{\bar{1}}\geq \lceil n/2 \rceil$. So,
if $v_{0}$ is even, then $\dim V_{\bar{1}}\geq \lceil n/2 \rceil$; if $v_{0}$ is odd, then $\dim V_{\bar{0}}\geq \lceil n/2 \rceil$.
By Lemma \ref{proposition5},  if $v_{0}$ is even, then $\dim V_{\bar{0}}\geq m+2$; if $v_{0}$ is odd, then $\dim V_{\bar{1}}\geq m+2$.
Therefore, $\dim V_{\bar{0}}\geq \min\{m+2, \lceil n/2 \rceil\}$ and $\dim V_{\bar{1}}\geq \min\{m+2, \lceil n/2 \rceil\}$.
Since  (\ref{eq12}) and (\ref{eq13}) define  a faithful representation
of $\mathfrak{h}_{m, n}$ with  super-dimension $(m+2, \lceil n/2 \rceil)$,
and (\ref{eq61}) and (\ref{eq62}) define a faithful representation
of $\mathfrak{h}_{m, n}$ with super-dimension $(\lceil n/2 \rceil, m+2)$,
we have
$$\mu_{0}(\frak{h}_{m,n})=\mu_{1}(\frak{h}_{m,n})=\min\{m+2, \lceil n/2 \rceil\}.
$$
It follows from  (\ref{10}) that
$$\mu_{0}^{*}({\mathfrak{h}_{m, n}})=\mu_{1}^{*}({\mathfrak{h}_{m, n}})=m+\lceil n/2 \rceil+2.
$$

Let $\mathfrak{g}=\mathfrak{h}_{n}$. By Lemma \ref{lem2}, we have $\dim \mathfrak{a}'_{\bar{1}}\geq 1$.
So, if $v_{0}$ is even, then $\dim V_{\bar{1}}\geq 1$; if $v_{0}$ is odd, then $\dim V_{\bar{0}}\geq 1$.
On the other hand, if $v_{0}$ is even, then $v_{0}\in V_{\bar{0}}$ and $\dim V_{\bar{0}}\geq 1$;
if $v_{0}$ is odd, then $v_{0}\in V_{\bar{1}}$ and $\dim V_{\bar{1}}\geq 1$.
Then by  (\ref{eq71}), we have
$$\mu_{0}(\mathfrak{h}_{n})=\mu_{1}(\mathfrak{h}_{n})=1.$$
 It follows from (\ref{10}) that
$$\mu_{0}^{*}(\mathfrak{h}_{n})=\mu_{1}^{*}(\mathfrak{h}_{n})=n+2.$$
\end{proof}
\begin{remark} Let $L$ be a Lie superalgebra and $\Pi$ the parity functor of the category of   $\mathbb{Z}_{2}$-graded vector spaces. It is well known that if $V$ is an $L$-module, then so is $\Pi(V)$ with respect to the original module action. Therefore, in general we have
$$
\mu_{0}(L)=\mu_{1}(L), \quad \mu_{0}^{*}(L)=\mu_{1}^{*}(L).
$$
This fact may be used to shorten the proofs of Theorems \ref{corollary4} and \ref{corollary5}.
\end{remark}

\noindent \textbf{Acknowledgements}

The authors are grateful to the anonymous referee for his/her valuable  comments and helpful suggestions. The first author was supported by the NSF of China (11171055, 11471090) and the NSF  of HLJ Province, China (A201412, JC201004)

\end{document}